\documentclass[12pt]{amsart}

\usepackage{amssymb, hyperref}

\usepackage[all]{xy}

\theoremstyle{plain}
\newtheorem{thm}{Theorem}
\newtheorem{theorem}{Theorem}[section]
\newtheorem{lemma}[theorem]{Lemma}
\newtheorem{proposition}[theorem]{Proposition}
\newtheorem{corollary}[theorem]{Corollary}

\theoremstyle{definition}

\newtheorem{remark}[theorem]{Remark}

\newtheorem{example}[theorem]{Example}

\newcommand\bG{{\mathbb G}}

\newcommand\bV{{\mathbb V}}

\newcommand\bZ{{\mathbb Z}}

\newcommand\bfD{\mathbf{D}}
\newcommand\bfF{\mathbf{F}}
\newcommand\bfHom{\mathbf{Hom}}
\newcommand\bfKer{\mathbf{Ker}}
\newcommand\bfR{\mathbf{R}}

\newcommand\cO{{\mathcal O}}
\newcommand\cP{{\mathcal P}}

\newcommand\fm{\mathfrak{m}}

\newcommand\rD{{\rm D}}

\newcommand\ab{{\rm ab}}

\newcommand\aff{{\rm aff}}

\newcommand\ant{{\rm ant}}

\newcommand\charac{{\rm char}}

\newcommand\diag{{\rm diag}}

\newcommand\gp{{\rm gp}}
\newcommand\gr{{\rm gr}}

\newcommand\id{{\rm id}}

\newcommand\pr{{\rm pr}}

\newcommand\red{{\rm red}}

\newcommand\sab{{\rm sab}}

\DeclareMathOperator\Ad{Ad}

\newcommand\Aut{{\rm Aut}}

\DeclareMathOperator\End{End}

\newcommand\Gal{{\rm Gal}}
\DeclareMathOperator\GL{GL}

\DeclareMathOperator\Hom{Hom}
\renewcommand\Im{{\rm Im}}

\DeclareMathOperator\Lie{Lie}

\DeclareMathOperator\Spec{Spec}
\newcommand\Stab{{\rm Stab}}
\DeclareMathOperator\Sym{Sym}

\numberwithin{equation}{section}

\title[On homomorphisms of algebraic groups]
{Homomorphisms of algebraic groups: representability and rigidity}

\author{Michel Brion}

\date{}

\address{Universit\'e Grenoble Alpes, 
Institut Fourier, CS 40700, 38058 Grenoble Cedex 9, France}

\begin{document}

\begin{abstract}
Given two algebraic groups $G$, $H$ over a field $k$, we
investigate the representability of the functor of morphisms 
(of schemes) $\bfHom(G,H)$ and the subfunctor of homomorphisms
(of algebraic groups)
$\bfHom_{\gp}(G,H)$. We show that $\bfHom(G,H)$ is represented 
by a group scheme, locally of finite type, if the $k$-vector space 
$\cO(G)$ is finite-dimensional; the converse holds if $H$ is 
not \'etale. When $G$ is linearly reductive and $H$ is smooth, 
we show that $\bfHom_{\gp}(G,H)$ is represented by a smooth 
scheme $M$; moreover, every orbit of $H$ acting by conjugation 
on $M$ is open.
\end{abstract}

\maketitle

\section{Introduction}
\label{sec:int}

The starting point of this article is the classification problem
for actions of an algebraic group $G$ on an algebraic variety $X$. 
When $X$ is proper over the ground field $k$,  its automorphism
functor is represented by a locally algebraic group $\Aut_X$ (i.e., 
a group scheme, locally of finite type), see \cite[Thm.~3.7]{MO}.  
Then the $G$-actions on $X$ correspond bijectively to the homomorphisms 
$f: G \to \Aut_X$, and the above problem is equivalent to classifying 
these homomorphisms up to conjugation by $\Aut(X) = \Aut_X(k)$.
This motivates the following questions: 

\begin{itemize}
\item 
Given an algebraic group $G$ and a locally algebraic group $H$, 
is the functor of homomorphisms $\bfHom_{\gp}(G,H)$ represented 
by a scheme $M$? 

\item
If so, $M$ is equipped with an action of $H$ by conjugation
on the target; how to describe the orbits?
\end{itemize}

When $G$ is of multiplicative type and $H$ is smooth 
and affine, the representability of $\bfHom_{\gp}(G,H)$ 
is due to Grothendieck; he showed in addition that 
the representing scheme $M$ is smooth and the morphism
\begin{equation}\label{eqn:graph}  
H \times M \longrightarrow M \times M, 
\quad (h,f) \longmapsto (h f h^{-1}, f) 
\end{equation}
is smooth as well (see \cite[Exp.~XI, Thm.~4.2, Cor.~5.1]{SGA3};
these results are obtained over an arbitrary base).  
As a consequence, for any field extension $K/k$ and any 
$f \in M(K) = \Hom_{K-\gp}(G_K,H_K)$, the orbit map $H_K \to M_K$, 
$h \mapsto h f h^{-1}$ is smooth,
and hence every orbit is open. This may be viewed as a rigidity 
property for actions of group schemes of multiplicative type: 
the only way to deform such an action is via conjugation on the
target.

For $G$ reductive and $H$ smooth affine, the representability 
of $\bfHom_{\gp}(G,H)$ was obtained by Demazure; in characteristic 
$0$, he also showed that the representing scheme $M$ is smooth 
(see \cite[Exp.~XXIV, Cor.~7.2.3, Prop.~7.3.1]{SGA3}; 
these results hold again over an arbitrary base). 
Further rigidity properties of $M$ were obtained by Margaux when 
$G$ is linearly reductive (see \cite{Margaux} and Remark
\ref{rem:margaux}).

Note that the above scheme $M$ is not necessarily of finite type;
for example, if $G = H = \bG_m$ then $M$ is the constant scheme
$\bZ_k$. Much more elaborate examples occur in recent work of
Lesieutre (see \cite[Lem.~12, Cor.~15]{Lesieutre}) and Dinh \& Oguiso
(see \cite[Lem.~4.5]{DO}): they constructed smooth complex 
projective varieties $X$ such that $\Aut(X)$ is discrete and has 
infinitely many conjugacy classes of involutions.

Yet if $\bfHom_{\gp}(G,H)$ is represented by a scheme $M$, 
then $M$ is locally of finite type. Indeed, when viewed as 
a functor from $k$-algebras to sets, $\bfHom_{\gp}(G,H)$ 
commutes with direct limits as a consequence of \cite[IV.8.8.2]{EGA}; 
thus, the assertion follows from the characterization of schemes 
locally of finite type in terms of their functors of points obtained 
in \cite[IV.8.14.2]{EGA}. Likewise, if the functor of morphisms
(of schemes) $\bfHom(G,H)$ is represented by a scheme $N$, 
then $N$ is locally of finite type.

The functors $\bfHom(G,H)$ and $\bfHom_{\gp}(G,H)$ 
usually contain more information than their sets of 
$K$-valued points for all field extensions $K/k$. For example, 
every morphism $f : \bG_{a,K} \to \bG_{m,K}$ is constant, but 
$\bfHom(\bG_a,\bG_m)$ is not representable, nor is 
$\bfHom_{\gp}(\bG_a,\bG_m)$ (see e.g. 
\cite[Exercise 1.1]{Milne}). Also, if $G$ and $H$ are linear
algebraic groups over an algebraically closed field $k$ of 
characteristic $0$, then the set of homomorphisms $\Hom_{\gp}(G,H)$
has a natural structure of affine ind-variety of finite dimension;
if in addition $G$ is unipotent, then one even gets an affine variety 
(by results of Furter \& Kraft, see \cite[Lem.~8.2.1, Prop.~8.4.1]{FK}). 
But in the latter case, $\bfHom_{\gp}(G,H)$ is representable if 
and only if $H$ is an extension of a finite group by a unipotent one.

With these motivations and examples in mind, we consider 
in this article the issues of representability of $\bfHom(G,H)$ 
and $\bfHom_{\gp}(G,H)$, where again $G$ is an algebraic group, 
and $H$ a locally algebraic group. Observe that $\bfHom(G,H)$ is
a group functor relative to pointwise multiplication in $H$,
and is the semi-direct product of the normal subgroup functor
$\bfHom(G,H; e_G \mapsto e_H)$ of pointed homomorphisms
by the group $H$ of constant morphisms (see Lemma \ref{lem:sdir}
for details). Also, it is easy to show that 
$\bfHom_{\gp}(G,H)$ is a closed subfunctor of $\bfHom(G,H)$ 
(Lemma \ref{lem:closed}). 

The case of an \'etale group scheme $H$ is easy as well: 
then $\bfHom(G,H)$ is represented by an \'etale scheme 
(Proposition \ref{prop:lae}). Thus, $\bfHom_{\gp}(G,H)$
is represented by an \'etale scheme as well.
So we may exclude this case in our first main result, 
which handles the representability of $\bfHom(G,H)$:

\begin{thm}\label{thm:rep}
Let $G$ be an algebraic group, and $H$ a locally algebraic group.
Assume that $H$ is not \'etale. Then $\bfHom(G,H)$ is represented 
by a locally algebraic group if and only if the vector space
$\cO(G)$ is finite-dimensional.
\end{thm}

This can be reformulated by using the affinization theorem
(see \cite[III.3.8.2]{DG}): for any algebraic group $G$, 
the affine scheme $G^{\aff} = \Spec \cO(G)$ is 
an algebraic group and the canonical morphism $G \to G^{\aff}$ is 
a faithfully flat homomorphism. Moreover, its kernel $N$ is 
smooth, connected, central in $G^0$ (in particular, commutative)
and satisfies $\cO(N) = k$; we say that $N$ is anti-affine.
As a consequence, $N$ is the largest anti-affine subgroup of
$G$; we denote it by $G_{\ant}$.  

Thus, Theorem \ref{thm:rep} asserts that $\bfHom(G,H)$ is
represented by a locally algebraic group if and only if $G$ 
is an extension of a finite group scheme by an anti-affine one. 
The proof begins with a reduction to the case where $G$ is 
anti-affine; we then show that $\bfHom(G,H; e_G \mapsto e_H)$ 
is equal to $\bfHom_{\gp}(G,H)$ and is represented by 
a form of some constant group scheme $\bZ^n_k$ 
(Proposition \ref{prop:rig}). 
For this, we use a rigidity lemma for anti-affine schemes 
(Lemma \ref{lem:rig}), a version of a result of C.~and 
F.~Sancho de Salas (see \cite[Thm.~1.7]{SS}). 

Our second main result gives a sufficient condition for the 
homomorphism functor to be representable. To state it, we 
introduce a variant of the classical notion of linear 
reductivity. We say that an algebraic group $G$ (possibly 
non-affine) is \emph{linearly reductive} if every $G$-module 
is semi-simple; this is equivalent to the affinization $G^{\aff}$ 
being linearly reductive. Examples of linearly reductive groups 
include group schemes of multiplicative type and semi-abelian
varieties; see the beginning of Section \ref{sec:prooflred} 
for more on this notion.

We may now state our second main result in a simplified version; 
see Theorem \ref{thm:sred} for the full, more technical statement.

\begin{thm}\label{thm:lred}
Let $G$ be a linearly reductive algebraic group, and $H$ a locally 
algebraic group. Then $\bfHom_{\gp}(G,H)$ is represented by a smooth 
scheme $M$. Moreover, the morphism (\ref{eqn:graph}) is smooth.
\end{thm}

Conversely, if the assertions of Theorem \ref{thm:lred}
hold for an algebraic group $G$ and all affine algebraic groups 
$H$, then $G$ is linearly reductive; see Remark \ref{rem:converse}.  
So this theorem yields a version of Grothendieck's representability 
and rigidity results mentioned above, which is close to optimal 
for group schemes over a field. We refer to \cite[Thm.~2]{Romagny} 
for a generalization of the representability theorem in its original 
setting of group schemes of multiplicative type over an arbitrary base.  

This article is organized as follows. Section \ref{sec:fun} contains
preliminary results on functors of (homo)morphisms; some of them 
are obtained in \cite[Exp.~I]{SGA3} in a much greater generality.
The tangent spaces to these functors are described in Section
\ref{sec:tangent} by using constructions and results from
\cite[Exp.~II]{SGA3}. Section \ref{sec:rep} collects representability
results for these functors, when restricted to various classes of group
schemes. In Section \ref{sec:proofrep}, we first prove the rigidity 
lemma mentioned above, and then deduce Theorem \ref{thm:rep}.
Theorem \ref{thm:sred} is stated and proved in the final Section 
\ref{sec:prooflred}, after some preliminary results on linear 
reductivity and a closely related notion of semi-reductivity.

\medskip

\noindent
{\bf Notation and conventions.}
We consider schemes over a field $k$ of characteristic 
$p \geq 0$, with separable closure $k_s$ and algebraic closure
$\bar{k}$. Morphisms and products are understood to be 
over $k$ unless otherwise stated. Schemes are assumed to be 
separated and locally of finite type throughout.

The structural morphism of a scheme $X$ is denoted by 
$\pi_X : X \to \Spec(k)$. Given a field extension $K/k$,
we denote by $X_K$ the $K$-scheme obtained from $X$ by 
the base change $\Spec(K) \to \Spec(k)$.

Group schemes are assumed to be locally algebraic in view of
our convention on schemes. 
Morphisms of group schemes will also be called homomorphisms.
The neutral element of a group scheme $G$ is denoted by $e_G$. 
An algebraic group is a group scheme of finite type.

We will freely use some of the functorial language of algebraic
geometry developed in \cite[I.1, I.2]{DG} (see also 
\cite[Chap.~VI]{EH}). We identify every scheme $S$ with 
its functor of points $h_S$.

\section{Hom functors}
\label{sec:fun}

We first recall some basic notions and results from 
\cite[Exp.~I, \S 7]{SGA3} in our special setting.
Given two schemes $X$, $Y$, we denote by $\bfHom(X,Y)$ the
contravariant functor from schemes to sets which sends every
scheme $S$ to $\Hom_S(X \times S, Y \times S)$, and every morphism 
of schemes $u : S' \to S$ to the pullback map
\[ \Hom_S(X \times S,Y \times S) \longrightarrow 
\Hom_{S'}(X \times S',Y \times S'). \]
We may identify $\bfHom(X,Y)(S)$ with $\Hom(X \times S,Y)$
by sending every morphism $f : X \times S \to Y$ to 
$(f, \pr_S) : X \times S \to Y \times S$. This identifies 
$\bfHom(X,Y)(u)$ with the map 
\[ \Hom(X \times S,Y) \longrightarrow \Hom(X \times S',Y),
\quad f \longmapsto f \circ (\id, u) \]
for any $u$ as above. As a consequence,  
$\bfHom(\Spec(k),Y)$ may be identified with $Y$.

The formation of $\bfHom(X,Y)$ commutes with base change
by field extensions $K/k$. Also, every morphism of schemes 
$\varphi : X' \to X$ induces a morphism of functors
\[ \varphi^* : \bfHom(X,Y) \longrightarrow \bfHom(X',Y) \]
via precomposition with $\varphi$. Likewise, every morphism 
of schemes $\psi : Y \to Y'$ induces a morphism of functors 
\[ \psi_* : \bfHom(X,Y) \longrightarrow \bfHom(X,Y') \]
via postcomposition with $\psi$. For any family of schemes
$(X_i)_{i \in I}$, the inclusions $X_i \to \coprod_{j \in I} X_j$
yield an isomorphism of functors
\[ \bfHom(\coprod_{i \in I} X_i,Y) 
\stackrel{\sim}{\longrightarrow} 
\prod_{i \in I} \bfHom(X_i,Y). \]
Likewise, for any family of schemes 
$(Y_i)_{i \in I}$, the projections 
$\pr_i : \prod_{j \in I} Y_j \to Y_i$ yield an isomorphism 
of functors
\[ \bfHom(X, \prod_{i \in I} Y_i)  
\stackrel{\sim}{\longrightarrow} 
\bfHom(X, \prod_{i \in I} Y_i). \]
We will also freely use the canonical isomorphism of functors
\[ \bfHom(X \times Y,Z) \stackrel{\sim}{\longrightarrow}
\bfHom(X,\bfHom(Y,Z)) \]
for any schemes $X$, $Y$, $Z$ (see 
\cite[Exp.~I, Prop.~1.7.1]{SGA3}). This identifies $\bfHom(Y,Z)$ 
with the Weil restriction functor $\bfR_{Y/k}(Z)$.

Next, recall the following result (a special case of  
\cite[I.2.7.5]{DG}):
 
\begin{lemma}\label{lem:closed}
Let $\psi : Y \to Y'$ be a closed immersion of schemes. Then
the morphism of functors 
$\psi_* : \bfHom(X,Y) \to  \bfHom(X,Y')$ is a closed immersion.
\end{lemma}

We now consider two morphisms of schemes
$\varphi_1,\varphi_2 : X' \to X$, and their equalizer 
$\bfKer(\varphi_1^*,\varphi_2^*)$, i.e., the subfunctor
of $\bfHom(X,Y)$ such that for any scheme $S$, the set
$\bfKer(\varphi_1^*,\varphi_2^*)(S)$ consists of the morphisms
$f: X \times S \to Y$ such that 
$f \circ (\varphi_1, \id) = f \circ (\varphi_2, \id)$.

\begin{lemma}\label{lem:equal}
With the above notation, $\bfKer(\varphi_1^*, \varphi_2^*)$
is a closed subfunctor of $\bfHom(X,Y)$.
\end{lemma}

\begin{proof}
We have a cartesian diagram of functors
\[ \xymatrix{
\bfKer(\varphi_1^*,\varphi_2^*) \ar[r] \ar[d] & \bfHom(X,Y) 
\ar[d]_{(\varphi_1^*, \varphi_2^*)}\\
\bfHom(X,Y) \ar[r]^-{\Delta_*} & \bfHom(X,Y \times Y), \\
} \]
where $\Delta : Y \to Y \times Y$ denotes the diagonal, and
$\bfHom(X,Y \times Y)$ is identified with
$\bfHom(X,Y) \times \bfHom(X,Y)$. Moreover, $\Delta_*$
is a closed immersion by Lemma \ref{lem:closed}; this yields
the assertion.
\end{proof}

\begin{lemma}\label{lem:descent}
Let $\varphi : X' \to X$ be a faithfully flat morphism 
of schemes, and $\pr_1,\pr_2 : X' \times_X X' \to X'$ 
the projections.

\begin{enumerate}

\item[{\rm (i)}]
$\varphi^*$ identifies $\bfHom(X,Y)$ with the equalizer 
$\bfKer(\pr_1^*,\pr_2^*)$.

\item[{\rm (ii)}] 
$\varphi^*$ is a closed immersion.

\end{enumerate}

\end{lemma}

\begin{proof}
(i) This holds by descent theory, see e.g. \cite[Thm.~2.55]{Vistoli} 
(note that $\varphi$ is locally of finite presentation in view of 
our standing assumption on schemes). 

(ii) This follows from (i) together with Lemma \ref{lem:equal}.
\end{proof}

In particular, we obtain:

\begin{corollary}\label{cor:constant}
The structural morphism of $X$ yields a closed immersion
\[ \pi_X^* : Y = \bfHom(\Spec(k),Y) \longrightarrow \bfHom(X,Y). \]
\end{corollary}

We may thus see $Y$ as the closed subfunctor of $\bfHom(X,Y)$
consisting of constant morphisms.

As a further direct consequence of Lemma \ref{lem:descent},
we record:

\begin{corollary}\label{cor:descent}
Let $G$ be a group scheme, and $\varphi: X' \to X$ 
a $G$-torsor for the fpqc topology. Then $\varphi^*$
identifies $\bfHom(X,Y)$ with the closed subfunctor of
$\bfHom(X',Y)$ consisting of $G$-invariant morphisms.
\end{corollary}

We now assume that $X$ (resp.~$Y$) is equipped with a
$k$-rational point $x$ (resp.~$y$). This yields a subfunctor
$\bfHom(X,Y; x \mapsto y)$ of $\bfHom(X,Y)$, such that 
for any scheme $S$, the set 
$\bfHom(X,Y; x \mapsto y)(S)$ consists of the morphisms
$f : X \times S \to Y$ which satisfy $f(x,s) = y$ identically
on $S$. In view of the cartesian diagram of functors
\[ \xymatrix{
\bfHom(X,Y; x \mapsto y) \ar[r] \ar[d] & \bfHom(X,Y) 
\ar[d]_{x^*}\\
\Spec(k) \ar[r]^-{y} & Y, \\
} \]
we see that $\bfHom(X,Y; x \mapsto y)$ is a closed subfunctor
of $\bfHom(X,Y)$.

In particular, for any group scheme $H$, we obtain a closed subfunctor 
$\bfHom(X,H; x \mapsto e_H)$ of $\bfHom(X,H)$.
Also, note that $\bfHom(X,H)$ is a group functor relative
to pointwise multiplication, and $\bfHom(X,H; x \mapsto e_H)$
is a normal subgroup functor. We also have the closed subfunctor 
$H$ of constant morphisms (Corollary \ref{cor:constant}); 
this is a subgroup functor as well.

\begin{lemma}\label{lem:sdir}
For any scheme $X$ equipped with a $k$-rational point $x$ and 
for any group scheme $H$, we have an isomorphism of group functors
\[ \bfHom(X,H) \simeq \bfHom(X,H; x \mapsto e_H) \rtimes H. \]
\end{lemma}

\begin{proof}
Let $S$ be a scheme, and $f \in \Hom(X \times S,H)$. 
Then we have $f = g h$, where 
$g \in \Hom(X \times S,H)$ sends $x \times S$ to $e_H$,
and $h \in \Hom(S,H)$: just take $h(s) = f(x, s)$ 
and $g = f h^{-1}$. Moreover, such a decomposition of $f$ 
is clearly unique. This yields the statement.
\end{proof}

Next, consider an exact sequence of group schemes
\[ 1 \longrightarrow N \stackrel{i}{\longrightarrow} H
\stackrel{q}{\longrightarrow} Q, \]
i.e., $i$, $q$ are homomorphisms, $i$ is a closed immersion, 
and its schematic image is the kernel of $q$. Then we readily 
obtain:

\begin{lemma}\label{lem:lex}
With the above notation, the sequence of group functors
\[ 1 \longrightarrow \bfHom(X,N) 
\stackrel{i_*}{\longrightarrow} \bfHom(X,H)
\stackrel{q_*}{\longrightarrow} \bfHom(X,Q) \] 
is exact. If $X$ is equipped with a $k$-rational point $x$, 
then the sequence of group functors
\[ 1 \to \bfHom(X,N; x \mapsto e_N) 
\stackrel{i_*}{\to} \bfHom(X,H; x \mapsto e_H)
\stackrel{q_*}{\to} \bfHom(X,Q; x \mapsto e_Q) \] 
is exact as well.
\end{lemma}

Given two group schemes $G$, $H$, we denote by 
$\bfHom_{\gp}(G,H)$ the subfunctor of 
$\bfHom(G,H ; e_G \mapsto e_H)$ consisting of homomorphisms. 
Clearly, the $H$-action on $\bfHom(G,H ; e_G \mapsto e_H)$ 
by conjugation on the target normalizes $\bfHom_{\gp}(G,H)$. 
If $H$ is commutative, then $\bfHom_{\gp}(G,H)$ is a subgroup
functor of the commutative group functor $\bfHom(G,H)$.

\begin{lemma}\label{lem:gp}
For any group schemes $G$, $H$, the subfunctor
$\bfHom_{\gp}(G,H)$ is closed in $\bfHom(G,H)$.
\end{lemma}

\begin{proof}
We adapt the argument of the proof of Lemma \ref{lem:equal}. 
Let $S$ be a scheme, and $f \in \Hom(G \times S, H)$.
Then $f$ is a homomorphism if and only if the diagram
\[ \xymatrix{
G \times G \times S \ar[r]^-{(\mu, \id) } \ar[d]_{(f, f,\id)} & 
G \times S \ar[d]_{(f, \id)} \\
H \times H \times S \ar[r]^-{(\nu, \id)} & H \times S \\
} \]
commutes, where $\mu$ (resp. $\nu$) denotes the multiplication
in $G$ (resp. $H$). This yields a cartesian diagram of functors
\[ \xymatrix{
\bfHom_{\gp}(G,H) \ar[r] \ar[d]_{\mu^*} & \bfHom(G,H) 
\ar[d]_{(\mu^*, \nu_* \circ \Delta)}\\
\bfHom(G \times G,H) \ar[r]^-{} & \bfHom(G \times G,H \times H), \\
} \]
where 
\[ \Delta = (\Delta_H)_* : \bfHom(G,H) \to \bfHom(G, H \times H) 
= \bfHom(G,H) \times \bfHom(G,H) \] 
denotes the diagonal. Since $\Delta_H$ is a closed immersion, 
this yields the statement in view of Lemma \ref{lem:closed}.
\end{proof}

\begin{lemma}\label{lem:hom}
Let $G_1$, $G_2$, $H$ be group schemes, and
$\alpha : G_1 \times G_2 \to G_2$ an action of $G_1$ on 
$G_2$ by automorphisms. Consider the semi-direct product
$G = G_1 \ltimes G_2$. Then the product of restriction functors
\[ \bfHom_{\gp}(G,H) \longrightarrow 
\bfHom_{\gp}(G_1,H) \times \bfHom_{\gp}(G_2,H) \]
is a closed immersion.
\end{lemma}

\begin{proof}
Let $S$ be a scheme, and $f: G \times S \to H$ a homomorphism.
Denote by $f_1$ (resp.~$f_2$) the restriction of $f$ to $G_1 \times S$ 
(resp.~$G_2 \times S$).
Then we have identically on $G_1 \times G_2 \times S$
\[ f(g_1 g_2,s) = f_1(g_1,s) f_2(g_2,s), \quad 
f_2(\alpha(g_1,g_2),s) = f_1(g_1,s) f_2(g_2,s) f_1(g_1,s)^{-1} \]
by the definition of the semi-direct product. Conversely, every pair
of homomorphisms $(f_1,f_2)$ satisfying the latter equality 
defines a unique homomorphism $f : G \times S \to H$, where the scheme
$G$ is identified with $G_1 \times G_2$. This yields
the assertion by arguing as in the proof of Lemma \ref{lem:equal}. 
\end{proof}

\section{Tangent spaces}
\label{sec:tangent}

We first recall the notion of tangent space for a functor
$\bfF$ from $k$-algebras to sets (see e.g.  \cite[VI.1.3]{EH}).
Denote by $D =  k[t]/(t^2)$ the algebra of dual numbers, so that
$D = k \oplus k \varepsilon$ where $\varepsilon^2 = 0$. 
The algebra homomorphism 
\[ \sigma : D \longrightarrow k, \quad \varepsilon \longmapsto 0 \]
yields a map $\bfF(\sigma) : \bfF(D) \to \bfF(k)$. The fiber of this map 
at $f \in \bfF(k)$ is the tangent space $T_f(\bfF)$; it is equipped
with an action of the multiplicative group $k^{\times}$ 
(the automorphism group of the $k$-algebra $D$).

More generally, each vector space $V$ yields 
a $k$-algebra $\bfD(V) = k \oplus \varepsilon V$, equipped with
the projection $\sigma_V : \bfD(V) \to k$ with kernel the ideal 
$\varepsilon V$ of square $0$. This defines a functor $\bfD$ from
vector spaces to algebras, satisfying $\bfD(k) = D$. 
For any two vector spaces $V$, $W$, we obtain a fiber product 
of algebras
\begin{equation}\label{eqn:fp}
\xymatrix{
\bfD(V \oplus W) \ar[r]^-{\bfD(\pr_V)} \ar[d]_-{\bfD(\pr_W)} 
& \bfD(V) \ar[d]_{\sigma_V} \\
\bfD(W) \ar[r]^-{\sigma_W} & k.\\
} \end{equation}
If $\bfF$ commutes with such fiber products, 
then $T_f(\bfF)$ has a natural structure of $k$-vector space.

Given two functors $\bfF_1$, $\bfF_2$ as above and a morphism
of functors $u : \bfF_1 \to \bfF_2$, the induced map 
$u(D): \bfF_1(D) \to \bfF_2(D)$ yields the differential
\[ T_f(u) : T_f(\bfF_1) \longrightarrow T_{u(f)}(\bfF_2) \] 
for any $f \in \bfF_1(k)$.
If $\bfF_1$, $\bfF_2$ commute with the fiber products
(\ref{eqn:fp}), then $T_f(u)$ is $k$-linear.

These notions may be applied to any contravariant functor 
from schemes to sets, and hence to $\bfHom(X,Y)$ where $X$, 
$Y$ are schemes. The resulting functor from algebras to sets 
commutes with the fiber products (\ref{eqn:fp}) in view of
\cite[Exp.~II, Cor.~3.11.2]{SGA3}. Moreover, for any 
$f \in \Hom(X,Y)$, we have a canonical isomorphism
of vector spaces
\begin{equation}\label{eqn:tan}
T_f \bfHom(X,Y) \simeq \Hom_Y(X,\bV(\Omega^1_Y)), 
\end{equation}
where $\Omega^1_Y$ denotes the $\cO_Y$-module of K\"ahler 
differentials of $Y$ over $k$, and 
$\bV(\Omega^1_Y) = \Spec_Y \Sym_{\cO_Y}(\Omega^1_Y)$ 
(see \cite[Exp.~II, Prop.~3.3, Cor.~3.11.3]{SGA3}).
Equivalently, we have canonical isomorphisms of vector spaces
\[ T_f \bfHom(X,Y) \simeq 
\Hom_{\cO_Y}(\Omega^1_Y,f_*(\cO_X)) \simeq 
\Hom_{\cO_X}(f^*(\Omega^1_Y),\cO_X). \]

Next, consider a group scheme $H$. By 
\cite[Exp.~I, Prop.~6.8.6]{SGA3}, the $\cO_H$-module
$\Omega^1_H$ is $H \times H$-equivariant, where $H \times H$
acts on $H$ by left and right multiplication. In particular,
$\Omega^1_H$ is equivariant relative to the right $H$-action.
Thus, there is a canonical isomorphism 
$\Omega^1_H \simeq \cO_H \otimes \Omega^1_H(e_H)$
in view of \cite[Exp.~I, Prop.~6.8.1]{SGA3}. Moreover,
we have canonical isomorphisms
$\Omega^1_H(e_H) \simeq \fm/\fm^2 \simeq \Lie(H)^*$,
where $\fm$ denotes the maximal ideal of the local ring
$\cO_{H,e_H}$, and $\Lie(H)$ stands for the Lie algebra. 
This yields a canonical isomorphism 
$\Omega^1_H \simeq \cO_H \otimes \Lie(H)^*$.
In view of the isomorphism (\ref{eqn:tan}), this yields in turn:

\begin{lemma}\label{lem:txh}
Let $X$ be a scheme, $H$ a group scheme, and $f : X \to H$ 
a morphism. Then there is a canonical isomorphism of vector spaces
\begin{equation}\label{eqn:i}
i : T_f \bfHom(X,H) \stackrel{\sim}{\longrightarrow}
\cO(X) \otimes \Lie(H). 
\end{equation}
\end{lemma}

\begin{remark}\label{rem:txh}
(i) With the above notation, we may view $\Lie(H)$ as 
the affine space $\Spec\,  \Sym(\fm/\fm^2)$; this identifies 
$\cO(X) \otimes \Lie(H)$ with $\Hom(X,\Lie(H))$.

If $X$ is equipped with a $k$-rational point $x$ and 
$f(x) = e_H$, then $i$ restricts to an isomorphism 
\begin{equation}\label{eqn:j} 
j : T_f \bfHom(X,H; x \mapsto e_H) 
\stackrel{\sim}{\longrightarrow} 
\Hom(X,\Lie(H); x \mapsto 0). 
\end{equation}

\noindent
(ii) By \cite[Exp.~II, \S 3.11]{SGA3}, the isomorphism 
(\ref{eqn:i}) may be interpreted as follows: consider
\[ \varphi \in T_f \bfHom(X,H) 
= \Hom_Y(X,\bV(\Omega^1_H)) = H^0(X,f^*(T_H)), \] 
where $T_H$ denotes the tangent bundle.
Let $S$ be a scheme, and $x \in X(S)$; then $f(x) \in H(S)$.
Thus, we may view $\varphi(x)$ as an $S$-point of $T_H$ 
above $f(x)$. Also, $T_H$ is equipped with a group scheme 
structure, the semi-direct product $\Lie(H) \rtimes H$
(see \cite[Exp.~II, \S 4.1]{SGA3}). Thus,
$\varphi(x) f(x)^{-1}$ is an $S$-point of the affine space 
$\Lie(H)$, and the assignment $x \mapsto \varphi(x) f(x)^{-1}$ 
gives back the isomorphism $i$.
\end{remark}

Next, consider a homomorphism of group schemes $f : G \to H$,
that is, $f \in \bfHom_{\gp}(G,H)(k)$. Then the $H$-action on 
$\bfHom_{\gp}(G,H)$ by conjugation yields a morphism of functors
\[ \gamma_f : H \longrightarrow \bfHom_{\gp}(G,H), \quad
h \longmapsto (g \mapsto h f(g) h^{-1}) \]
that we may see as the orbit map associated with $f$.
Since $\gamma_f(e_H) = f$, we have the differential 
\[ T_{e_H}(\gamma_f) :  T_{e_H}(H) \longrightarrow 
T_f \bfHom_{\gp}(G,H) \subset T_f \bfHom(G,H)  \] 
that we will view as a map $\Lie(H) \to \Hom(G,\Lie(H))$
by using the isomorphism (\ref{eqn:i}).

\begin{lemma}\label{lem:bz}
Keep the above notation and assumptions.

\begin{enumerate}

\item[{\rm (i)}] The tangent space $T_f \bfHom_{\gp}(G,H)$ 
is the subspace $Z^1(G,\Lie(H))$ of $\Hom(G,\Lie(H))$
consisting of the morphisms 
$\varphi : G \to \Lie(H)$ such that
$\varphi(g_1 g_2) = \varphi(g_1) + \Ad(f(g_1)) \varphi(g_2)$
identically on $G \times G$.

\item[{\rm (ii)}] The image of the differential 
$T_{e_H}(\gamma_f)$ is the subspace $B^1(G,\Lie(H))$ 
of $\Hom(G,\Lie(H))$ consisting of the morphisms 
$g \mapsto \Ad(f(g))x - x$, where $x \in \Lie(H)$.

\end{enumerate}

\end{lemma}

\begin{proof}
(i) This follows from \cite[Exp.~II, Prop.~4.2]{SGA3}.

(ii) We have $\gamma_f(h) f(g)^{-1} = h f(g) h^{-1} f(g)^{-1}$
identically on $G \times H$. In view of Remark \ref{rem:txh},
it follows that $T_{e_H}(\gamma_f)(x)(g) = x - \Ad(f(g))x$
for any $x \in \Lie(H)$ and any schematic point $g$ of $G$.

\end{proof}

\begin{remark}\label{rem:H1}
(i) The space $Z^1(G,\Lie(H))$  consists of the $1$-cocycles 
of the Hochschild complex $C^*(G,\Lie(H))$, where $\Lie(H)$ 
is a $G$-module via $\Ad \circ f$; moreover, the subspace 
$B^1(G,\Lie(H))$ consists of the $1$-coboundaries
(see \cite[Exp.~I, \S 5.1]{SGA3} or \cite[II.3]{DG}). Thus, we have 
\[ Z^1(G,\Lie(H))/B^1(G,\Lie(H)) = H^1(G,\Lie(H)), \]
the first cohomology group of this module.

With the notation and assumptions of Lemma \ref{lem:bz},
it follows that the map
\[ T_{e_H}(\gamma_f) : 
T_{e_H}(H) \longrightarrow T_f \bfHom_{\gp}(G,H) \]
is surjective if and only if $H^1(G,\Lie(H)) = 0$.

(ii) Most results on cohomology groups of a group scheme $G$
are obtained under the assumption that $G$ is affine. 
When $G$ is an algebraic group, this entails no loss of generality 
in view of the affinization theorem recalled in the introduction.
Indeed, the pullback map $\cO(G^{\aff}) \to \cO(G)$ is an isomorphism. 
Moreover, the representation $\Ad \circ f : G \to \GL(\Lie(H))$ factors 
uniquely through a representation of $G^{\aff}$. Therefore, $\Lie(H)$ is
a $G^{\aff}$-module, and the pullback maps
\[ Z^1(G^{\aff},\Lie(H)) \to Z^1(G,\Lie(H)), \;
B^1(G^{\aff},\Lie(H)) \to B^1(G,\Lie(H)) \] 
are isomorphisms. So we obtain an isomorphism
\[ H^1(G^{\aff},\Lie(H)) \stackrel{\sim}{\longrightarrow}
H^1(G,\Lie(H)) \]
(which extends to all cohomology groups of all $G$-modules).
\end{remark}

Next, we obtain a key smoothness result:

\begin{lemma}\label{lem:smooth}
Let $G$ be an algebraic group, and $H$ a smooth group scheme. 
Assume that $\bfHom_{\gp}(G,H)$ is represented by a scheme $M$. 
Then the following conditions are equivalent:

\begin{enumerate}

\item[{\rm (i)}] The morphism (\ref{eqn:graph}) 
\[ \gamma : H \times M \longrightarrow M \times M, 
\quad (h,f) \longmapsto (h f h^{-1}, f) \] 
is smooth.

\item[{\rm (ii)}] For any field extension $K/k$ and any
$f \in M(K) = \Hom_{K-\gp}(G_K,H_K)$, the morphism 
\[ \gamma_f : H_K \longrightarrow M_K,  \quad
h \longmapsto (g \mapsto h f(g) h^{-1}) \]
is smooth.

\item[{\rm (iii)}] For any field extension $K/k$ and any 
$f \in M(K)$, we have 
\[ H^1(G_K, \Lie(H_K)) = 0,  \]
where $\Lie(H_K)$ is a $G_K$-module via $\Ad \circ f$.

\end{enumerate}

These conditions hold whenever $G^{\aff}$ is linearly reductive.

\end{lemma}

\begin{proof}
(i)$\Rightarrow$(ii) Just observe that $\gamma$ is a morphism of
$M$-schemes relative to the second projections; moreover,
$\gamma_f$ is obtained from $\gamma$ by base change via
$f : \Spec(K) \to M$.

(ii)$\Rightarrow$(i) This follows from the above observation
together with the fiberwise criterion for smoothness
(see \cite[IV.17.8.2]{EGA}).

(ii)$\Rightarrow$(iii) Since $\gamma_f$ is smooth at $e_H$,
its differential at this point is surjective. This yields
the assertion in view of Lemma \ref{lem:bz}.

(iii)$\Rightarrow$(ii) We may assume that $K = k$ is algebraically
closed. Then $H$ is a disjoint union of $H(k)$-cosets of the neutral
component $H^0$, and hence we may further assume that $H$ 
is a smooth connected algebraic group. 
Also, $f \in M(k)$ and the morphism $\gamma_f : H \to M$
has a surjective differential at $e_H$ by Lemma \ref{lem:bz} again.
We now adapt to this setting some standard considerations for
actions of smooth algebraic groups on schemes of finite type
(recall that $M$ is locally of finite type).

Consider the schematic image $X$ of $\gamma_f$; this is a closed
integral subscheme of $M$, stable under $H(k)$ and hence 
under $H$ (see \cite[II.5.3.2]{DG}). Moreover, $\gamma_f$
factors uniquely through a dominant morphism 
$\varphi: H \to X$, equivariant for the action of $H$ by left
multiplication on itself. By generic flatness (see
\cite[IV.6.9.1]{EGA}), there exists a dense open subscheme
$U$ of $X$ such that the pullback $\varphi^{-1}(U) \to U$ is flat.
Since the $H(k)$-translates of $\varphi^{-1}(U)$ cover $H$,
it follows that $\varphi$ is flat.

The fiber of $\varphi$ at $e_H$ is the isotropy subgroup
scheme $\Stab_H(f)$, which is smooth in view of the 
vanishing of $H^1(G,\Lie(H))$ (see \cite[II.5.2.8]{DG}).
Thus, $\varphi$ is smooth at $e_H$ (see \cite[IV.17.5.1]{EGA})
and hence everywhere by equivariance. So the image of 
$\varphi$ (the orbit $H \cdot f$) is open in $X$ and smooth
(see \cite[IV.2.4.6, IV.6.8.3]{EGA}); 
in particular, $H \cdot f$ is locally closed in $M$. 
Also, $T_f(H \cdot f) = T_f(M)$ by Lemma \ref{lem:bz} again.

We now claim that the natural homomorphism of local rings
\[ u : \cO_{M,f} \longrightarrow \cO_{H \cdot f,f} \] 
is an isomorphism. Indeed, $u$ is clearly surjective.
Consider the associated graded homomorphism
\[ \gr(u) : \gr(\cO_{M,f}) \longrightarrow 
\gr(\cO_{H \cdot f,f}). \]
The right-hand side is a polynomial ring in $n$ generators
of degree $1$, where $n = \dim(H \cdot f)$. Also,
$\gr(u)$ induces an isomorphism on subspaces of degree $1$.
As the left-hand side is generated in degree $1$, it
follows that $\gr(u)$ is an isomorphism. As a consequence,
$u$ is injective, proving the claim.

By this claim, $H \cdot f$ contains an open neighborhood 
of $f$ in $M$. Using equivariance again, it follows that 
$H \cdot f$ is open in $M$. Since the morphism 
$H \to H \cdot f$ is smooth, so is $\gamma_f$.

Finally, if $G^{\aff}$ is linearly reductive,
then Remark \ref{rem:H1}(ii) and \cite[II.3.3.7]{DG}
yield the vanishing of $H^1(G_K, \Lie(H_K))$ for any field 
extension $K/k$.  
\end{proof}

\section{Representability: first steps}
\label{sec:rep}

\subsection{Morphisms to \'etale group schemes}
\label{subsec:etale}

For any group scheme $G$, we denote by $G^0$ its neutral 
component, i.e., the connected component of $e_G$. Recall that 
$G^0$ is an algebraic group, and is the kernel of the homomorphism 
\[ \gamma: G \longrightarrow \pi_0(G), \] 
where $\pi_0(G)$ is the \'etale group scheme of connected components. 
Moreover, $\gamma$ is faithfully flat (see \cite[II.5.1.8]{DG}), 
and hence a $G^0$-torsor.

\begin{proposition}\label{prop:lae}
Let $G$ be an algebraic group, and $H$ an \'etale group scheme. 

\begin{enumerate}

\item[{\rm (i)}] 
The pullback $\gamma^*: \bfHom(\pi_0(G),H) \to \bfHom(G,H)$ 
is an isomorphism.

\item[{\rm (ii)}] 
The functor $\bfHom(G,H)$ is represented by an \'etale group 
scheme $N$. If $H$ is finite, then $N$ is finite as well.

\end{enumerate}

\end{proposition}

\begin{proof}
(i) By Corollary \ref{cor:descent}, $\gamma^*$ identifies
$\bfHom(\pi_0(G),H)$ with the subfunctor of $G^0$-invariants 
in $\bfHom(G,H)$. So it suffices to show
that for any scheme $S$, every morphism $G \times S \to H$ is
invariant under $G^0$. For this, we may assume $k$ algebraically
closed by descent. Then the group schemes $\pi_0(G)$ and
$H$ are constant; moreover, $G = \coprod_{i \in I} g_i G^0$
for some family $(g_i)_{i \in I}$ of $G(k)$, where $I = \pi_0(G)(k)$.
So we may further assume that $G$ is connected. Then it suffices
to show that every morphism $f : G \times S \to H$ that sends
$e_G \times S$ to $e_H$ is constant (Lemma \ref{lem:sdir}). 
We may assume that $S$ is connected; then $G \times S$ is connected 
as well (see e.g.  \cite[Exp.~V, Lem.~2.1.2]{SGA3}). The schematic 
fiber of $f$ at $e_H$ is open and closed in $G \times S$, and contains
$e_G \times S$; so this fiber is the whole $G \times S$ as desired.

(ii) In view of (i), we may replace $G$ with $\pi_0(G)$, and hence
assume that $G$ is finite and \'etale. Using Galois descent, we may 
further assume that $G$ is constant. Then $\bfHom(G,H)$ 
is the constant group scheme associated with the group of maps 
$G(k) \to H(k)$. 
\end{proof}

\begin{corollary}\label{cor:con}
Let $G$ be a connected algebraic group, and $H$ a group scheme.
Then the inclusion of $H^0$ in $H$ induces isomorphisms
\[ \bfHom(G,H^0;e_G \mapsto e_H) \stackrel{\sim}{\longrightarrow} 
\bfHom(G,H;e_G \mapsto e_H), \]
 \[ \bfHom_{\gp}(G,H^0) \stackrel{\sim}{\longrightarrow} 
 \bfHom_{\gp}(G,H). \]
\end{corollary}

\begin{proof}
By Proposition \ref{prop:lae}, the natural morphism 
$\pi_0(H) \to \bfHom(G,\pi_0(H))$ is an isomorphism. Thus, 
$\bfHom(G,\pi_0(H); e_G \mapsto e_{\pi_0(H)}) = \Spec(k)$.
Together with the exact sequence 
$1 \to H^0 \to H \to \pi_0(H)$ and with Lemma \ref{lem:lex},
this yields the assertions.
\end{proof}

\subsection{Two finiteness notions}
\label{subsec:FT}

We introduce finiteness notions on schemes, which will be
very convenient for proving Theorem \ref{thm:sred}. 

We say that a scheme $X$ satisfies (FT) (resp.~(AFT)) if every 
connected component of $X$ is of finite type (resp.~affine
of finite type). Since $X$ is locally of finite type, its
connected components are open (see \cite[I.Cor.~6.1.9]{EGA}).
Thus, (FT) (resp.~(AFT)) is equivalent to $X$ being a sum
(in the sense of \cite[I.3.1]{EGA}) of schemes of finite type
(resp.~affine of finite type).

We now record some basic properties of these notions,
with no attempt for exhaustivity.

\begin{lemma}\label{lem:ft}

Consider a group scheme $G$ and two schemes $X$, $Y$.

\begin{enumerate}

\item[{\rm (i)}] 
$G$ satisfies (FT). Also, $G$ satisfies (AFT) if and only if 
$G^0$ is affine.

\item[{\rm (ii)}] 
If $X$ is \'etale, then it satisfies (AFT).

\item[{\rm (iii)}] 
If $X$ and $Y$ satisfy (FT) (resp.~(AFT)), then so does 
$X \times Y$.

\item[{\rm (iv)}] If $X$ satisfies (FT) (resp.~(AFT)), then
so does the $K$-scheme $X_K$ for any field extension $K/k$.

\item[{\rm (v)}] If there exists a finite Galois extension
$K/k$ such that $X_K$ satisfies (FT) (resp.~(AFT)),
then so does $X$.

\item[{\rm (vi)}] If $Y$ is a closed subscheme of $X$, 
and $X$ satisfies (FT) (resp.~(AFT)), then so does $Y$.

\end{enumerate}

\end{lemma}

\begin{proof}
(i) The assertion on (FT) follows from the ``theorem of the
neutral component'' (see \cite[II.5.1.1]{DG} or 
\cite[Exp.~VIA, Cor.~2.4.1]{SGA3}). 

If $G$ satisfies (AFT), then $G^0$ is clearly affine.
Conversely, assume that $G^0$ is affine, and consider
a connected component $C$ of $G$. Then $C_{\bar{k}}$
is the sum of finitely many translates of $G^0_{\bar{k}}$
and hence is affine. By descent, it follows that $C$ is
affine as well.

(ii) Just recall that every connected component of $X$ 
is of the form $\Spec(K)$ for some finite (separable)
field extension $K/k$.

(iii) This follows from the fact that finite products commute
with sums (see \cite[I.3.2.8]{EGA}), and preserve the properties 
of being of finite type (resp.~affine of finite type); 
see \cite[I.Prop.~6.3.4]{EGA}.

(iv) This is checked by a similar argument.

(v) Assume that $X_K$ satisfies (FT) for $K/k$ finite
Galois with group $\Gamma$. Then $\Gamma$ acts
on $X_K$, and permutes its connected components.
Thus, $X_K$ is a sum of $\Gamma$-stable schemes
of finite type. By Galois descent for subschemes, 
it follows that $X$ is a sum of schemes of finite type, 
i.e., it satisfies (FT). The same argument works for (AFT).

(vi) Let $Y'$ be a connected component of $Y$. Then $Y'$ 
is a closed subscheme of a unique connected component 
$X'$ of $X$. So the assertion follows from 
\cite[I.Prop.~6.3.4]{EGA} again.
\end{proof}

\subsection{Morphisms from finite group schemes}
\label{subsec:finite}

We first record an easy observation:

\begin{lemma}\label{lem:fla}
Let $X$ be a finite scheme, and $H$ a group scheme. Then 
the functor $\bfHom(X,H)$ is represented by a group scheme.
\end{lemma}

\begin{proof}
Recall the isomorphism $\bfHom(X,H) \simeq \bfR_{X/k}(H)$,
where $\bfR_{X/k}$ denotes the Weil restriction functor.
By Lemma \ref{lem:ft}, every finite subset of the underlying
topological space of $H$ is contained in an open affine
subscheme of finite type. The representability of 
$\bfR_{X/k}(H)$ by a scheme (locally of finite type) follows 
from this by \cite[I.1.6.6]{DG} and its proof. 
\end{proof}

Next, let $G$, $H$ be group schemes, where $G$ is finite. 
Combining Lemmas \ref{lem:gp}, \ref{lem:ft} and \ref{lem:fla}, 
we see that $\bfHom_{\gp}(G,H)$ is represented by a scheme $M$ 
satisfying (FT). If $H$ is algebraic, then $M$ is of finite type.

Also, recall from \cite[II.4.7.1]{DG} that a group scheme $G$
is \emph{infinitesimal} if $G$ is finite and $e_G$ is its unique
point.

\begin{proposition}\label{prop:finite}
With the above notation and assumptions, the scheme $M$ is affine 
of finite type under either of the following conditions:

\begin{enumerate}

\item[{\rm (i)}] $G$ is infinitesimal.

\item[{\rm (ii)}] $H$ is affine.

\item[{\rm (iii)}] $H$ is connected.

\end{enumerate}

\end{proposition}

\begin{proof}
(i) We may assume that $\charac(k) = p > 0$.
For any scheme $X$, we denote by $F_X: X \to X^{(p)}$
the relative Frobenius morphism and by
\[ F^n_X : X \longrightarrow X^{(p^n)} \] 
its $n$th iterate, where $n$ is a positive integer. 
If $X$ is a group scheme, then $F^n_X$ is a homomorphism; 
we denote by $X_n$ its kernel (the $n$th Frobenius kernel). 

By assumption, we have $G = G_n$ for $n \gg 0$. For any scheme
$S$ and any homomorphism $f : G \times S \to H$, we have a
commutative diagram
\[ \xymatrix{
G \times S \ar[r]^-{f} \ar[d]_{(\pi_G, F^n_S)} & 
H \ar[d]_{F^n_H}\\
S^{(p^n)}  \ar[r]^-{f^{(n)}} & H^{(p^n)}. \\
} \]
Thus, we have identically on $G \times S$
\[ F^n_H(f(g,s)) = f^{(n)}(F^n_S(s))  = F^n_H f(e_G,s) 
= e_{H^{(p^n)}}. \]
So $f$ factors uniquely through $H_n$. Thus,
$\bfHom_{\gp}(G,H_n) \stackrel{\sim}{\longrightarrow} \bfHom_{\gp}(G,H)$. 
As $G$ and $H_n$ are finite, we may view
$\bfHom_{\gp}(G,H_n)$ as the functor of Hopf algebra 
homomorphisms $\cO(H_n) \to \cO(G)$. This is a closed subfunctor
of the functor of morphisms of $k$-modules $\cO(H_n) \to \cO(G)$,
and the latter functor is represented by an affine space.

(ii) Since Hom functors commute with base change, we may assume
$k$ algebraically closed. Then $G \simeq I \rtimes F$, where $I = G^0$
is infinitesimal, and $F \simeq \pi_0(G)$ is the constant group scheme 
associated with $G(k)$ (see e.g.  \cite[II.5.2.4]{DG}). In view of 
Lemma \ref{lem:hom}, we may thus assume in addition
that $G$ is either infinitesimal or constant. 
In the former case, we conclude by (i). In the latter case, 
$\bfHom(G,H) \simeq H^n$, where $n$ denotes the order of $G$. 
Thus, $\bfHom(G,H)$ is affine of finite type, and hence so is 
$\bfHom_{\gp}(G,H)$.

(iii) By arguing as in the proof of (ii), we reduce to the case where
$k$ is algebraically closed and $G$ is constant of order $n$.
Then $g^n = e_G$ identically on $G$. Thus, for any scheme $S$,
every homomorphism $f : G \times S \to H$ factors uniquely 
through the schematic fiber at $e_H$ of the $n$th power map of $H$.
Denoting this fiber by $H[n]$, it follows that $\bfHom_{\gp}(G,H)$
is isomorphic to a closed subscheme of 
$\bfHom(G,H[n]) \simeq H[n]^n$. So it suffices to show that 
$H[n]$ is affine.

As $H$ is connected, there exists an exact sequence of algebraic
groups
\begin{equation}\label{eqn:chevalley}
1 \longrightarrow N \longrightarrow H 
\stackrel{q}{\longrightarrow} A \longrightarrow 1, 
\end{equation}
where $N$ is affine and $A$ is an abelian variety
(see  \cite[Lem.~IX.2.7]{Raynaud}). Thus, 
$H[n]$ is a closed subscheme of the pullback $q^{-1}(A[n])$.
Since $A[n]$ is a finite group scheme and $q$ is affine, 
this yields the desired statement.
\end{proof}

\begin{example}\label{ex:ell}
Let $G$ be the constant group scheme of order $2$, and 
$H = E \rtimes G$, where $E$ is an elliptic curve on which
$G$ acts by $\pm 1$. Then $H$ is a smooth proper non-connected
algebraic group, and one readily checks that $\bfHom_{\gp}(G,H)$
is represented by  $H[2] \simeq E[2] \coprod E$.
So $\bfHom_{\gp}(G,H)$ is not affine.
\end{example}

\begin{example}\label{ex:alpha}
Assume that $p > 0$ and consider the infinitesimal group scheme 
$\alpha_p$, the kernel of the Frobenius endomorphism of $\bG_a$.
For any group scheme $H$, the functor $\bfHom_{\gp}(\alpha_p,H)$
is represented by the fiber at $0$ of the $p$th power map of $\Lie(H)$,
as follows from \cite[Exp.~VII, Thm.~7.2]{SGA3} or \cite[II.7.4.2]{DG}. 
For example, $\bfHom_{\gp}(\alpha_p,\bG_a)$ is represented 
by the affine line.
\end{example}

\subsection{Homomorphisms from tori}
\label{subsec:tori}

The following result is a version of Gro\-thendieck's representability
theorem stated in the introduction (see \cite[Exp.~XI, Thm.~4.2]{SGA3}):

\begin{proposition}\label{prop:tori}
Let $G$ be a torus, and $H$ a group scheme. Then the functor 
$\bfHom_{\gp}(G,H)$ is represented by a scheme satisfying (AFT).
\end{proposition}

\begin{proof}
We first consider the case where $H$ is an abelian variety.
We claim that for any scheme $S$, every homomorphism
$f : G \times S \to H$ is constant.

To show this, we adapt a classical rigidity argument. 
By descent, we may assume that $k$ is algebraically closed.
Also, we may assume that $S$ is connected. Choose
a positive integer $n$ prime to $p$. Then the $n$-torsion
subgroups $G[n]$, $H[n]$ are finite and constant.
For any $g \in G[n](k)$,  the morphism
$S \to H[n]$, $s \mapsto f(g,s)$ is constant. Choose
$s_0 \in S(k)$, then we have $f(g,s) = f(g,s_0)$ identically
on $G[n] \times S$. Since the family of the $G[n]$ for $n$
as above is schematically dense in $G$, the family of the
$G[n] \times S$ is schematically dense in $G \times S$
(see \cite[IV.11.10.6]{EGA}). Thus, $f(g,s) = f(g,s_0)$
identically on $G \times S$. Also, the morphism $G \to H$, 
$g \mapsto f(g,s_0)$ is a homomorphism; it follows that 
$f(g,s_0) = e_H$ identically on $G$, since $G$ is 
a torus, and $H$ an abelian variety. This yields the claim.

We now consider the general case. By Corollary \ref{cor:con},
we may assume that $H$ is a connected algebraic group. 
Thus, $H$ lies in an exact sequence of the form 
(\ref{eqn:chevalley}). Using the above claim together
with Lemma \ref{lem:lex}, we may further assume that $H$
is affine. Then $H$ is isomorphic to a closed subgroup scheme
of some general linear group $\GL_n$. In view of Lemmas 
\ref{lem:closed}, \ref{lem:gp} and \ref{lem:ft}, we may thus reduce 
to the case where $H = \GL_n$. Also, since the torus $G$ is split 
by a finite Galois extension of $k$, we may assume that 
$G \simeq \bG_m^r$ by using Lemma \ref{lem:ft} again. 
For any homomorphism $f : G \times S \to \GL_n$, 
the $\cO_S$-module $\cO_S^n$ is a $G$-module via $(f, \id)$, 
and hence is the direct sum of its weight submodules (see 
\cite[Exp.~I, Prop.~4.7.3]{SGA3}). It follows that 
$\bfHom_{\gp}(G,H)$ is represented by the scheme
\[  \coprod_{(n_1, \ldots,n_s)} S_s \times
\GL_n/\GL_{n_1} \times \cdots \times \GL_{n_s}, \]
where $(n_1,\ldots,n_s)$ runs over the tuples of positive integers 
with sum $n$, and $S_s \subset (\bZ^r)^s$ denotes the set of 
$s$-tuples of pairwise distinct weights (viewed as a constant scheme). 
Since every homogeneous space 
$\GL_n/\GL_{n_1} \times \cdots \times \GL_{n_s}$ 
is affine of finite type, this completes the proof.
\end{proof}

\begin{corollary}\label{cor:red}
Let $G$ be a reductive algebraic group, and $H$ a group scheme.
Then the functor $\bfHom_{\gp}(G,H)$ is represented by a scheme 
$M$ satisfying (FT). If $H$ is affine, then $M$ satisfies (AFT).
\end{corollary}

\begin{proof}
By Corollary \ref{cor:con}, we may assume that $H$ is a connected
algebraic group.
Also, we may choose a maximal torus $T$ of $G$. Then 
$\bfHom_{\gp}(T,H)$ is representable by Proposition \ref{prop:tori};
moreover, the morphism of functors
\[ u : \bfHom_{\gp}(G,H) \longrightarrow \bfHom_{\gp}(T,H) \]
is relatively representable by a morphism locally of finite 
presentation, as a consequence of \cite[Exp.~XXIV, Prop.~7.2.1]{SGA3}. 
Thus, $\bfHom_{\gp}(G,H)$ is representable by a scheme $M$ 
(locally of finite type).

To show that $M$ satisfies (FT), recall that $T_K$ is split for some
finite Galois field extension $K/k$. Using Galois descent and
Lemma \ref{lem:ft} (v), we may therefore assume that $T$ is split. 
We now use \cite[Exp.~XXIV, Cor.~7.1.9]{SGA3},
which asserts that the above morphism $u$ satisfies every property 
$\cP$ of morphisms which is stable by composition and base change, 
and holds for closed immersions and for the structural morphism $\pi_H$. 
In view of the assumption on $H$ and \cite[I.Prop.~6.3.4]{EGA}, 
we may take for $\cP$ the property of being of finite type. 
Also, $\bfHom_{\gp}(T,H)$ satisfies (FT) by Proposition \ref{prop:tori} 
again. It follows readily that $\bfHom_{\gp}(G,H)$ satisfies (FT). 
The proof for (AFT) is obtained similarly by taking for $\cP$ 
the property of being affine of finite type.
\end{proof}

\subsection{Morphisms from abelian varieties}
\label{subsec:abelian}

The following result is a consequence of the rigidity lemma
proved in the next section (Lemma \ref{lem:rig}). We provide
a short and direct proof.

\begin{proposition}\label{prop:abelian}

Let $G$ be an abelian variety, and $H$ a group scheme.

\begin{enumerate}

\item[{\rm (i)}] $H$ has a largest abelian subvariety $H_{\ab}$. 

\item[{\rm (ii)}] We have 
\[ \bfHom_{\gp}(G,H_{\ab}) =
\bfHom(G,H_{\ab}; e_G \mapsto e_H) =
\bfHom(G,H; e_G \mapsto e_H)  \]
and this functor is represented by a commutative \'etale 
group scheme.

\end{enumerate}

\end{proposition}

\begin{proof}
(i) Clearly, we may assume that $H$ is connected, 
and hence an algebraic group. Let $A$, $B$ be abelian 
subvarieties of $H$, and assume that $A$ has maximal
dimension among all such subvarieties. Then $A$, $B$
are both contained in $H_{\ant}$, and hence centralize
each other (since every anti-affine group is commutative). 
Thus, the morphism $A \times B \to H$, 
$(a,b) \mapsto a b$ is a homomorphism, and its image
is an abelian subvariety $C$ of $H$. By maximality,
we have $C = A$, and hence $B \subset A$.

(ii) By Corollary \ref{cor:con}, we may again assume that 
$H$ is a connected algebraic group. Then the scheme 
$H$ is quasi-projective (see \cite[Cor.~VI.2.6]{Raynaud}); also,
$G$ is projective. As a consequence, the functor 
$\bfHom(G,H)$ is represented by a quasi-projective scheme
$M$ (see \cite[p.~268]{Grothendieck}). Thus, the closed subfunctor
$\bfHom(G,H; e_G \mapsto e_H)$ is represented by a 
closed subscheme $N$ of $M$. Moreover, the formation
of $N$ commutes with base change by field extensions.

We now show that $N$ is \'etale. For this, we may assume that 
$k$ is algebraically closed in view of \cite[IV.17.7.3]{EGA}.
Since $N$ is locally of finite type, it suffices to show
that $T_f(N) = 0$ for any $f \in N(k)$. But this follows 
from the isomorphism (\ref{eqn:j}), since every morphism 
$G \to \Lie(H)$ is constant.

We have a chain of closed subfunctors
\[ \bfHom_{\gp}(G,H_{\ab}) \subset 
\bfHom(G,H_{\ab}; e_G \mapsto e_H) \subset 
\bfHom(G,H;e_G \mapsto e_H). 
\]
Moreover, the resulting inclusions of sets of $K$-points
are equalities for any algebraically closed field $K$,
in view of (i) and \cite[\S 4, Cor.~1]{Mumford}. 
Since $\bfHom(G,H; e_G \mapsto e_H)$ is represented
by an \'etale scheme, these inclusions of subfunctors are
equalities. As $\bfHom_{\gp}(G,H_{\ab})$ is a commutative 
group functor, this yields the assertion. 
\end{proof}

It is shown in \cite[Thm.~7.2]{LS} that the formation of 
$H_{\ab}$ commutes with base change by field extensions; 
we will not need this fact.

\section{Proof of Theorem \ref{thm:rep}}
\label{sec:proofrep}

We begin with an easy observation:

\begin{lemma}\label{lem:nonrig}
Let $G$ be a group scheme.

\begin{enumerate}

\item[{\rm (i)}] If $\bfHom(G,H)$ is representable for some non-\'etale
group scheme $H$, then the vector space $\cO(G)$ is finite-dimensional.

\item[{\rm (ii)}] If 
$\bfHom(G,\bG_a; e_G \mapsto 0) = \bfHom_{\gp}(G,\bG_a)$,
then $G$ is anti-affine.

\end{enumerate}

\end{lemma}

\begin{proof}
(i) Consider the constant morphism $f: G \to H$ with image
$e_H$. Then $T_f \bfHom(G,H) \simeq \cO(G) \otimes \Lie(H)$
by Lemma \ref{lem:txh}; also, $\Lie(H) \neq 0$ as $H$ is not
\'etale. If $\bfHom(G,H)$ is representable, then its tangent 
space at $f$ is finite-dimensional, hence the assertion.

(ii) Every $f \in \cO(G)$ such that $f(e_G) = 0$
satisfies $f(g_1 g_2) = f(g_1) + f(g_2)$ identically on $G \times G$.
Applying this to $f^2$, we obtain that $2 f(g_1) f(g_2) = 0$ 
identically. If $p \neq 2$, it follows that $f = 0$; thus, $\cO(G) = k$. 
If $p = 2$ we consider $f^3$ and argue similarly.
\end{proof}

Next, we obtain a key rigidity result:

\begin{lemma}\label{lem:rig}
Let $X$ be a geometrically reduced scheme of finite type 
such that $\cO(X) = k$. Let $Y$ be a geometrically connected 
scheme, and $f : X \times Y \to Z$ a morphism of schemes. 
Assume that there exist  $x_0 \in X(\bar{k})$ and 
$y_0 \in Y(\bar{k})$ such that $f(x,y_0) = f(x_0,y_0)$ for all 
$x \in X(\bar{k})$. Then $f$ factors through the projection
$\pr_Y: X \times Y \to Y$.
\end{lemma}

\begin{proof}
By fpqc descent, it suffices to show that the base change
$f_{\bar{k}}$ factors through $(\pr_Y)_{\bar{k}}$. 
Thus, we may assume that $k$ is algebraically closed. 

Let $W$ be the pullback of $\diag(Z)$ under the morphism
\[ X \times Y \longrightarrow Z \times Z, \quad
(x,y) \longmapsto (f(x,y), f(x_0,y)). \]
Then the equalizer $W$ is a closed subscheme of $X \times Y$.

Consider the subset $Y'$ of $Y$ consisting of those $y$ such that
$X \times \{ y \} \subset W$ as sets; equivalently,
$X \times \{ y \} \subset W$ as schemes, since $X$ is reduced. 
We claim that $Y'$ is closed in $Y$. Indeed, $X \times Y' \subset W$ 
as sets, and hence $\overline{X \times Y'} \subset W$. Since the 
projection $X \times Y \to Y$ is open, we have 
$\overline{X \times Y'} = X \times \overline{Y'}$. 
Thus, $\overline{Y'} = Y'$, proving the claim.

By this claim and the connectedness of $Y$, it suffices to show 
that every $y \in Y'(k)$ admits an open neighborhood $U = U(y)$ 
such that $X \times U \subset W$ as schemes. 

Let $z = f(x_0,y)$; then $z \in Z(k)$, and $f$ induces a morphism
\[ f_n : X \times Y_n \longrightarrow Z_n \]
for any positive integer $n$, where $Y_n$ (resp.~$Z_n$) denotes 
the $n$th infinitesimal neighborhood of $y$ in $Y$ (resp.~of
$z$ in $Z$). Since $Z_n$ is finite, $f_n$ is given by an algebra 
homomorphism
\[ f_n^{\#} : \cO(Z_n) \longrightarrow \cO(X \times Y_n). \]
But we have 
$\cO(Y_n) \stackrel{\sim}{\longrightarrow} \cO(X \times Y_n)$
since $X$ is of finite type, $\cO(X) = k$ and $Y_n$ is finite
(see \cite[I.2.2.6]{DG}). Thus,
$f_n$ factors through $\pr_{Y_n} : X \times Y_n \to Y_n$. 
So $f(x,y) = f(x_0,y)$ identically on $X \times Y_n$, i.e.,
$X \times Y_n \subset W$. By Krull's intersection theorem,
the family $(Y_n)_{n\geq 1}$ is schematically dense in an open 
neighborhood $U$ of $y$ in $Y$. Thus, the family 
$(X \times Y_n)_{n \geq 1}$ is schematically dense in 
$X \times U$ in view of \cite[IV.11.10.6]{EGA}. Since $W$
is a closed subscheme of $X \times Y$, we obtain 
that $X \times U \subset W$ as desired.
\end{proof}

As mentioned in the introduction, the above result is a version of  
\cite[Thm.~1.7]{SS}. The proof presented there (and again in 
\cite[\S 4.3]{BSU} and \cite[Lem.~3.3.3]{Bri17}) requires some minor 
corrections, e.g., there is a confusion in the final step between 
density and schematic density.

We may now obtain the following result, a version of
\cite[Prop.~5.1.4]{BSU}:

\begin{proposition}\label{prop:rig}
Let $G$ be an anti-affine algebraic group, and $H$ a group scheme. 
Then $\bfHom_{\gp}(G,H) = \bfHom(G,H; e_G \mapsto e_H)$ and 
this functor is represented by a form of $\bZ^n_k$ for some
integer $n \geq 0$.
\end{proposition}

\begin{proof}
To show the equality of functors, we may assume that $k$ is
algebraically closed. We now adapt a classical argument
(see \cite[p.~43]{Mumford}). 
Let $S$ be a connected scheme, and $f : G \times S \to H$
a morphism that sends $e_G \times S$ to $e_H$. 
Consider the morphism
\[ F : G \times G \times S \longrightarrow H, \quad
(x,y,s) \longmapsto f(xy,s) f(y,s)^{-1} f(x,s)^{-1}. \]
Then $F(x,e_G,s) = F(e_G,x,s) = e_H$ for all $x \in G$ and $s \in S$.
Moreover, $G \times S$ is connected. By Lemma \ref{lem:rig},
it follows that $F(x,y,s) = e_H$ identically on $G \times G \times S$,
i.e., $f$ is a homomorphism.

It remains to show that $\bfHom_{\gp}(G,H)$ is represented by a
form of some $\bZ^n_k$. We first treat the case where $k$ 
is separably closed. Consider again a connected scheme $S$, and let
\[ f : G \times S \longrightarrow H \] 
be a homomorphism. Choose an $s_0 \in S(\bar{k})$; then the morphism
of $\bar{k}$-schemes
\[ G_{\bar{k}} \times_{\bar{k}} S_{\bar{k}} \longrightarrow H_{\bar{k}}, 
\quad (x,s) \longmapsto f(x,s) f(x,s_0)^{-1} \]
sends $G_{\bar{k}} \times_{\bar{k}} s_0$ to $e_H$, and hence factors 
through the projection 
$G_{\bar{k}} \times_{\bar{k}} S_{\bar{k}} \to S_{\bar{k}}$
by Lemma \ref{lem:rig}. Since $f(e_G,s) = e_H$ identically on $S$, 
it follows that $f(g,s) = f(g,s_0)$ identically on 
$G_{\bar{k}} \times_{\bar{k}} S_{\bar{k}}$. Thus, 
$f_{\bar{k}}$ factors through the projection 
$G_{\bar{k}} \times_{\bar{k}} S_{\bar{k}} \to G_{\bar{k}}$.
By fpqc descent, it follows that $f$ factors through the projection
$G \times S \to G$. This shows that $\bfHom_{\gp}(G,H)$ is represented 
by the constant scheme $\Hom_{\gp}(G,H)_k$. Since 
$\Hom_{\gp}(G,H^0_{\ant}) \stackrel{\sim}{\to} \Hom_{\gp}(G,H)$,
this yields the assertion in view of \cite[Lem.~1.5(ii)]{Bri09}.

The case of an arbitrary field $k$ follows by Galois descent.
Indeed, for any $f \in \Hom_{k_s-\gp}(G_{k_s},H_{k_s})$, there exist 
a finite Galois extension $K/k$ and a homomorphism of $K$-group
schemes $\varphi : G_K \to H_K$ such that $f = \varphi_{k_s}$,
since $f$ factors through the algebraic group $H^0_{k_s}$.
\end{proof}

\medskip \noindent
{\bf Completion of the proof of Theorem \ref{thm:rep}.} 
If $\bfHom(G,H)$ is representable, then the vector space $\cO(G)$ 
is finite-dimensional by Lemma \ref{lem:nonrig}. Conversely, assume
that $\cO(G)$ is finite-dimensional. Then we have an exact sequence
of algebraic groups
\[ 1 \longrightarrow G_{\ant} \longrightarrow G 
\longrightarrow F \longrightarrow 1, \]
where $F$ is finite. By \cite[Thm.~1.1]{Bri15}, there exists 
a finite subgroup scheme $F' \subset G$ such that $G = G_{\ant} F'$.
This yields an exact sequence of algebraic groups
\[ 1 \longrightarrow F'' \longrightarrow G_{\ant} \rtimes F' 
\longrightarrow G \longrightarrow 1, \]
where $F''$ is finite as well. In view of Corollary \ref{cor:descent},
we may thus assume that $G = G_{\ant} \rtimes F'$.
Then $G \simeq G_{\ant} \times F'$ as schemes, and hence
\[ \bfHom(G,H) \simeq \bfHom(G_{\ant},\bfHom(F',H)). \]
Since $\bfHom(F',H)$ is represented by a group scheme
(Lemma \ref{lem:fla}), we may further assume that $G$ is anti-affine.
Then the assertion follows from Lemma \ref{lem:sdir} and 
Proposition \ref{prop:rig}.

\section{Proof of Theorem \ref{thm:lred}}
\label{sec:prooflred}

We begin with some observations and structure results on
the class of linearly reductive groups. We will also consider 
the class of \emph{semi-reductive} groups: we say that an algebraic 
group is semi-reductive if its affinization $G^{\aff}$ 
is an extension of a finite group scheme by a reductive group
scheme. Both classes turn out to be closely related.

The \emph{affine} linearly reductive groups are well-understood:
if $p = 0$, they are exactly the extensions of finite group schemes
by reductive group schemes, i.e., the affine semi-reductive groups 
(see e.g.  \cite[IV.3.3.3]{DG}). If $p > 0$, then the affine linearly 
reductive groups are exactly the extensions 
\[ 1 \longrightarrow H \longrightarrow G \longrightarrow F
\longrightarrow 1, \]
where $F$ is a finite group scheme of order prime to $p$
and $H$ is a connected group scheme of multiplicative type
(see \cite[IV.3.3.6]{DG}). Moreover, $H$ has a largest
subtorus $T$, with Cartier dual being the quotient 
of the Cartier dual of $H$ by its torsion subgroup (see 
\cite[IV.1.3]{DG}). As a consequence, $T$ is characteristic
in $H$, and hence normal in $G$. So the affine linearly
reductive groups are exactly the extensions of finite linearly
reductive groups by tori. Thus, every affine linearly reductive 
group is semi-reductive, but the converse fails (if $p > 0$ again). 

As a consequence, \emph{every linearly reductive group is
semi-reductive; the converse holds if and only if $p = 0$.} 
If $p > 0$ then the linearly reductive groups are exactly
the extensions of finite linearly reductive groups by 
semi-abelian varieties (as follows from the fact that 
every anti-affine group is a semi-abelian variety,
see \cite[Prop.~2.2]{Bri09}). In particular, the
smooth connected linearly reductive groups are
exactly the semi-abelian varieties.

We now discuss the behavior of both classes under base 
change by a field extension $K/k$. By 
\cite[Prop. ~3.2]{Margaux}, an affine algebraic group $G$ is
linearly reductive if and only if so is $G_K$. Clearly, this 
invariance property also holds for anti-affine algebraic groups.
In view of the affinization theorem, it follows that 
\emph{an algebraic group $G$ is linearly reductive if and only 
if so is $G_K$.} 

Also, \emph{if an algebraic group $G$ is semi-reductive, 
then so is $G_K$.} 
The converse holds if $K/k$ is separable algebraic (by Galois 
descent), but fails for purely inseparable extensions. 
For example, if $k$ is separably closed but not algebraically
closed, then there exists a non-split extension of $\bG_m$ 
by the infinitesimal group scheme $\alpha_p$, see 
\cite[Exp.~XVII, 5.9 c)]{SGA3}. As every such extension
splits over $\bar{k}$, this yields an example of 
an affine algebraic group $G$ such that $G_{\bar{k}}$ is 
semi-reductive, but $G$ is not.

Next, we discuss the behavior of both classes under taking
quotients, normal subgroup schemes and extensions. Clearly, 
\emph{every quotient of a linearly reductive group is linearly 
reductive. The class of semi-reductive groups is also stable 
under quotients,} since so are the classes of anti-affine, 
reductive and finite group schemes. 

By \cite[Prop.~3.4]{Margaux}, the affine linearly reductive 
groups are also stable by normal subgroup schemes and extensions. 
But the class of linearly reductive groups is not
stable under normal subgroup schemes. To see this if $p = 0$,
consider a non-trivial extension $G$ of an elliptic curve $E$ by
$\bG_a$; then $G$ is anti-affine (see \cite[Prop.~2.3]{Bri09}),
but of course $\bG_a$ is not. If $p > 0$, let 
$H = \alpha_p \rtimes \bG_m$, where $\bG_m$ acts on 
$\alpha_p$ as its automorphism group.  Also, let $E$ be 
a supersingular elliptic curve; then $\alpha_p$ is isomorphic 
to a subgroup of $E$.  Consider the pushout
\[ \xymatrix{
1 \ar[r] & \alpha_p \ar[r] \ar[d] & H \ar[r] \ar[d] 
& \bG_m \ar[r]  \ar[d] & 1 \\
1 \ar[r] & E \ar[r]  & G \ar[r] & \bG_m \ar[r]  & 1. \\
} \]
Then $G$ is linearly reductive, but $H$ is not. These examples also
show that the class of semi-reductive groups is not stable under
normal subgroup schemes.

\emph{The class of linearly reductive groups is stable by extensions}.
Indeed, by a standard argument, an algebraic group $G$ is
linearly reductive if and only if the fixed point functor 
$V \mapsto V^G$ (from finite-dimensional $G$-modules to vector 
spaces) is exact. Moreover, for any exact sequence 
of algebraic groups $1 \to N \to G \to Q \to 1$ and any $G$-module 
$V$, we have $V^G = (V^N)^Q$.

In particular, if $p = 0$ then the class of semi-reductive groups 
is stable under extensions. This fails if $p > 0$ in view of 
the following:

\begin{example}\label{ex:nonext}
Let $V$ be a vector space of finite dimension $n \geq 1$. 
Consider the relative Frobenius morphism 
$F_{V/k} : V \to V^{(p)}$ and denote by $U$ its kernel;
then $U$ is an infinitesimal unipotent subgroup scheme of $V$, 
normalized by the natural action of $\GL(V)$. Form the 
semi-direct product $G = U \rtimes \GL(V)$; then $G$ is
clearly an extension of a reductive group scheme by a 
finite group scheme. But there is no exact sequence
\[ 1 \longrightarrow R \longrightarrow G \longrightarrow F
\longrightarrow 1, \]
where $R$ is reductive and $F$ is finite. Otherwise,
$R$ is the reduced subscheme $G^0_{\red}$, and hence
$R = \GL(V)$. As $R \triangleleft G$, it follows that 
$\GL(V)$ centralizes $U$, a contradiction. 
\end{example}

Next, we obtain several criteria for semi-reductivity:

\begin{proposition}\label{prop:sr}
Let $G$ be an algebraic group. Consider the conditions:

\begin{enumerate}

\item[{\rm (i)}] $G$ is semi-reductive.

\item[{\rm (ii)}] There exists an exact sequence of algebraic groups
\begin{equation}\label{eqn:long} 
1 \longrightarrow F_1 \longrightarrow G_1 \times G_2 
\longrightarrow G \longrightarrow F_2 \longrightarrow 1, 
\end{equation}
where $F_1$, $F_2$ are finite, $G_1$ is anti-affine, and
$G_2$ is reductive. 

\item[{\rm (iii)}] The affinization $G^{\aff}$ is linearly reductive.

\item[{\rm (iv)}] $G$ is smooth and $G_{\bar{k}}$ has no
non-trivial smooth connected unipotent normal subgroup.

\end{enumerate}

Then {\rm (iii)}$\Rightarrow${\rm (i)}$\Leftrightarrow${\rm (ii)}
and {\rm (iv)}$\Rightarrow${\rm (i)}. If $p =0$ then 
{\rm (i)}, {\rm (ii)}, {\rm (iii)} are equivalent.
If $p > 0$ then {\rm (i)}, {\rm (ii)}, {\rm (iv)} are equivalent
for $G$ smooth.
\end{proposition}

\begin{proof}
(ii)$\Rightarrow$(i) Cut the long exact sequence (\ref{eqn:long}) 
in two short exact sequences
\begin{equation}\label{eqn:ex}
1 \longrightarrow F_1 \longrightarrow G_1 \times G_2
\longrightarrow G_0 \longrightarrow 1, \quad
1 \longrightarrow G_0 \longrightarrow G \longrightarrow F_2
\longrightarrow 1,
\end{equation}
where $G_0$ is the largest smooth connected normal subgroup
scheme of $G$.
Denote by $N$ the schematic image of $G_1$ in $G_0$;
then $N = (G_0)_{\ant}$ in view of
\cite[Lem.~3.3.6]{Bri17}. Thus, $N = G_{\ant}$ is normal in $G$, 
and $G/N = G^{\aff}$ is an extension of the finite group scheme 
$F_2$ by the schematic image of $G_2$ in $G_0$; this image 
is a reductive group scheme. This yields the assertion.

(iii)$\Rightarrow$(i) Just recall that every 
linearly reductive group is semi-reductive.

(iv)$\Rightarrow$(i) 
It suffices to show that $G^0$ is semi-reductive. We claim that 
$G^0$ satisfies (iv). Indeed, $G^0$ is smooth since so is $G$.
Consider the largest smooth connected unipotent normal subgroup 
$U$ of $G^0_{\bar{k}}$; then $U$ is normalized by $G(\bar{k})$, 
and hence $U \triangleleft G_{\bar{k}}$. So $U$ is trivial, 
proving the claim.

Thus, we may assume that $G$ is connected. By the Rosenlicht
decomposition (see e.g. \cite[Thm.~5.5.1]{Bri17}), there exists 
a smooth connected affine algebraic $\bar{k}$-group 
$H \triangleleft G_{\bar{k}}$ such that 
$G_{\bar{k}} = (G_{\bar{k}})_{\ant} H$. Since 
$(G_{\bar{k}})_{\ant}$ is central in $G_{\bar{k}}$, we see that
the unipotent radical of $H$ is trivial, i.e., $H$ is reductive. 
So $G_{\bar{k}}/(G_{\bar{k}})_{\ant}$ is reductive as well.
Since the formation of $G_{\ant}$ commutes with field extensions,
it follows that $G/G_{\ant}$ is reductive.

For the remaining implications, we treat the cases $p = 0$ and
$p > 0$ separately as we use the structure of anti-affine groups,
which differs in both cases (see \cite[\S 2]{Bri09}).

Assume that $p = 0$. Then (i)$\Rightarrow$(iii) follows from 
the linear reductivity of affine semi-reductive groups.
We now show (i)$\Rightarrow$(ii). 
By assumption, we have an exact sequence
\begin{equation}\label{eqn:ex2} 
1 \longrightarrow G_{\ant} \longrightarrow G^0
\longrightarrow R \longrightarrow 1, 
\end{equation}
where $R$ is reductive. Consider again the Rosenlicht decomposition 
$G^0 = G_{\ant} G^0_{\aff}$. 
By the main result of \cite{Mostow}, there exists a Levi decomposition 
$G^0_{\aff} = R_u(G^0_{\aff}) \rtimes L$,
where $L$ is reductive, and $R_u$ denotes the unipotent radical. 
In view of (\ref{eqn:ex2}),  
it follows that $R_u(G^0_{\aff})$ is the largest unipotent subgroup 
of $G_{\ant}$, and $G^0 = G_{\ant} L$. This yields an exact sequence
\[ 1 \longrightarrow M \longrightarrow L 
\longrightarrow R \longrightarrow 1, \]
where $M = G_{\ant} \cap L$ is central in $L$, and hence of multiplicative 
type. So there exists a reductive subgroup scheme $L'$ of $L$ such that 
$L = M L'$ and $M \cap L'$ is finite. Thus, $G^0 = G_{\ant} L'$ and 
$G_{\ant} \cap L'$ is finite. Equivalently, we have an exact sequence
\[ 1 \longrightarrow G_{\ant} \cap L' \longrightarrow G_{\ant} \times L' 
\longrightarrow G^0 \longrightarrow 1, \]
which yields the assertion.

Next, we assume that $p > 0$, and show that (i)$\Rightarrow$(ii). Recall 
that $G_{\ant}$ is a semi-abelian variety (see \cite[Prop.~2.1]{Bri09}). 
We may assume that there is an exact sequence
$1 \longrightarrow G_{\ant} \longrightarrow G \longrightarrow R
\longrightarrow 1$,
where $R$ is reductive.  In particular, $G$ is smooth and connected;
hence its derived subgroup $\rD(G)$ is smooth, connected and affine.
Also, $R = T \rD(R)$ for some central torus $T$. Thus, the pullback of
$T$ in $G$ is the largest normal semi-abelian variety $G_{\sab}$; 
moreover,  $G = G_{\sab} \rD(G)$, and $\rD(G)$ is reductive. 
It follows that $G = G_{\ant} L$ for some reductive subgroup scheme $L$, 
and we conclude as above.

Still assuming that $p > 0$, we show that (i)$\Rightarrow$(iv)
when $G$ is smooth. We may assume that $k$ is algebraically
closed. Let $U$ be a smooth connected unipotent normal subgroup 
of $G$; then $U \cap G_{\ant}$ is finite. As the unipotent radical
of $G/G_{\ant} = G^{\aff}$ is trivial, this yields the assertion.
\end{proof}

(As already mentioned, the implication (i)$\Rightarrow$(iii)
fails if $p > 0$. Also, the implication (i)$\Rightarrow$(iv) 
fails if $p = 0$, as shown again by the example of a non-trivial 
extension of an elliptic curve by the additive group.)

We now obtain a version of Theorem \ref{thm:lred} in arbitrary
characteristic:

\begin{theorem}\label{thm:sred}
Let $G$ be a semi-reductive algebraic group, and $H$ a group scheme. 
Then $\bfHom_{\gp}(G,H)$ is represented by a scheme $M$
satisfying (FT). If $H$ is affine, then $M$ satisfies (AFT). 
Also, the morphism (\ref{eqn:graph}) is smooth if $G$ is linearly 
reductive and $H$ is smooth.
\end{theorem}

\begin{proof}
If $G$ is finite, then the first assertion follows from Lemmas 
\ref{lem:gp}, \ref{lem:ft} and \ref{lem:fla}. For an arbitrary $G$, 
consider again the two exact sequences (\ref{eqn:ex}), 
where $G_1$ is anti-affine, $G_2$ is reductive, and
$F_1$, $F_2$ are finite. Thus, there exists a finite subgroup 
scheme $F \subset G$ such that $G = G_0 F$ (see 
\cite[Thm.~1]{Bri15}). By arguing as in the proof of Theorem 
\ref{thm:rep} and using Corollary \ref{cor:descent} 
and Lemma \ref{lem:hom}, we may thus assume that $G = G_0$.
In particular, $G$ is connected. In view of Corollary 
\ref{cor:con}, we may further assume that $H$ is connected.
 
We now use again Corollary \ref{cor:descent} and Lemma
\ref{lem:hom} to reduce to the case where $G$ is either anti-affine 
or reductive. In the latter case, we conclude by Corollary 
\ref{cor:red}; in the former case, we use Proposition 
\ref{prop:rig}.

This shows the assertion about Property (FT); the proof
of the assertion about (AFT) is similar and left to the
reader. The final assertion follows from Lemma \ref{lem:smooth}.
\end{proof}

\begin{remark}\label{rem:margaux}
Assume that $k$ is algebraically closed of characteristic $0$.
Consider a semi-reductive group $G$ and a group scheme $H$. 
By Theorem \ref{thm:sred}, for any $m \in M(k)$, the orbit 
$H^0 \cdot m$ is open in the connected component of $m$ in $M$. 
Since every such orbit is connected, it follows that the connected 
components of $M$ are exactly the $H^0$-orbits of $k$-rational 
points. So the set of $H$-orbits in $M$ is in one-to-one
correspondence with the set of $k$-rational points of the quotient 
$\pi_0(M)/\pi_0(H)$, where $\pi_0(M)$ denotes the (constant) scheme 
of connected components. Also, since $k = \bar{k}$, the above set 
of $H$-orbits in $M$ may be identified with the orbit space 
$M(k)/H(k) = \Hom_{\gp}(G,H)/H(k)$.  As a consequence, we have 
\begin{equation}\label{eqn:vm} 
\Hom_{\gp}(G/H)/H(k) \stackrel{\sim}{\longrightarrow}
 \Hom_{K-\gp}(G_K,H_K)/H(K) 
\end{equation}
for any algebraically closed field extension $K/k$.

Next, assume that $k$ is algebraically closed of characteristic
$p > 0$, and consider a linearly reductive group $G$ and 
a smooth group scheme $H$.  Then all the above results hold
without any change, in view of Lemma \ref{lem:smooth} and 
Theorem \ref{thm:sred} again. 
The bijection (\ref{eqn:vm}) gives back the main result of 
\cite{Margaux} (Theorem 1.1; see also \cite[Prop.~10]{Vinberg}).
\end{remark}

\begin{remark}\label{rem:converse}
Theorem \ref{thm:sred} has a partial converse: let $G$ be an algebraic 
group and assume that for any smooth affine algebraic group $H$, the functor
$\bfHom_{\gp}(G,H)$ is represented by a scheme $M$ such that the 
morphism (\ref{eqn:graph}) is smooth. Then $G$ is linearly reductive.

Indeed, we have $H^1(G,\Lie(H)) = 0$ for any $f \in \Hom_{\gp}(G,H)$, 
where $G$ acts on $\Lie(H)$ via $\Ad \circ f$ (Lemma \ref{lem:smooth}).  
As a consequence, $H^1(G,\End(M)) = 0$ for any
finite-dimensional representation $f : G \to \GL(M)$. But every
finite-dimensional $G$-module $V$ is a summand of some $\End(M)$:
just take $M = V \oplus k$, where $G$ acts trivially on $k$. 
Thus, $H^1(G,V) = 0$ for any such $V$; this yields the assertion
by \cite[II.3.3.7]{DG} together with Remark \ref{rem:H1}.
\end{remark}

\medskip \noindent
{\bf Acknowledgments.} Many thanks to Mathieu Florence, 
Matthieu Romagny and Antoine V\'ezier for their careful reading 
of preliminary versions and for very helpful comments. 
Example \ref{ex:nonext} was suggested by Mathieu Florence. 
I thank Matthieu Romagny and Peng Du for pointing out a confusion 
between density and schematic density in the original proof of 
the rigidity lemma. Also, thanks to Philippe Gille for asking
a number of stimulating questions and for drawing my attention 
on the rigidity result of Margaux.  Finally, I thank the referee for
valuable remarks and comments.

\bibliographystyle{amsalpha}

\end{document}